\documentclass[a4paper,11pt]{amsart}
\usepackage{amsmath,amssymb,amsfonts,amsthm}
\usepackage{latexsym,mathrsfs}
\usepackage{graphicx}
\usepackage[all]{xy}
\input xypic
\usepackage{color}
\usepackage[backref=false]{hyperref}
\hypersetup{colorlinks=true,linkcolor=red,citecolor=blue}
\usepackage[OT2,T1]{fontenc}
\usepackage[utf8]{inputenc}
\usepackage{lmodern}
\usepackage{microtype}
\usepackage{enumerate}
\usepackage{tikz-cd}
\usepackage{bm}
\theoremstyle{plain}
\newtheorem{theorem}[subsection]{{\bf Theorem}}
\newtheorem*{theorem*}{{\bf Theorem}}
\newtheorem{corollary}[subsection]{{\bf Corollary}}
\newtheorem*{corollary*}{{\bf Corollary}}
\newtheorem{proposition}[subsection]{{\bf Proposition}}
\newtheorem{lemma}[subsection]{{\bf Lemma}}

\theoremstyle{definition}

\theoremstyle{remark}

\numberwithin{equation}{subsection}
\DeclareMathOperator{\HH}{H}
\DeclareMathOperator{\K}{K}
\DeclareMathOperator{\B}{B}

\DeclareMathOperator{\Hom}{Hom}

\DeclareMathOperator{\res}{res}
\DeclareMathOperator{\cor}{cor}
\DeclareMathOperator{\conj}{conj}

\DeclareBoldMathCommand{\bbot}{\bot}

\DeclareSymbolFont{cyrletters}{OT2}{wncyr}{m}{n}
\DeclareMathSymbol{\Sha}{\mathalpha}{cyrletters}{"58}

\newcommand{\QZ}{\mathbb{Q}/\mathbb{Z}}
\newcommand{\Z}{\mathbb{Z}}

\begin{document}
\title[Local control and Bogomolov multipliers]{Local control and Bogomolov multipliers of finite groups}
\author{Primo\v z Moravec}
\address{{
Faculty of  Mathematics and Physics, University of Ljubljana,
and Institute of Mathematics, Physics and Mechanics,
Slovenia}}
\email{primoz.moravec@fmf.uni-lj.si}
\thanks{ORCID: \url{https://orcid.org/0000-0001-8594-0699}. The author acknowledges the financial support from the Slovenian Research Agency, research core funding No. P1-0222, and projects No. J1-3004, N1-0217, J1-2453, J1-1691.}
\date{\today}
%%%%%%%%%%%%%%%%%%%%%%%%%%%%%%%%%%%%%%%%%%%%%%%%%%%%%%%%%%%%%%%%%%%%%
\begin{abstract}
\noindent
 We show that if a Sylow $p$-subgroup of a finite group $G$ is nilpotent of class at most $p$, then the $p$-part of the Bogomolov multiplier of $G$ is locally controlled.
\end{abstract}
%%%%%%%%%%%%%%%%%%%%%%%%%%%%%%%%%%%%%%%%%%%%%%%%%%%%%%%%%%%%%%%%%%%%%%%
\maketitle
%%%%%%%%%%%%%%%%%%%%%%%%%%%%%%%%%%%%%%%%%%%%%%%%%%%%%%%%%%%%%%%%%%%%%%%
\section{Introduction}
\label{s:intro}

\noindent
In this note we prove the following result:

\begin{theorem}
    \label{thm:bog}
    Let $G$ be a finite group and $P$ a Sylow $p$-subgroup of $G$. If the nilpotency class of $P$ does not exceed $p$, then the Bogomolov multipliers of $G$ and $N_G(P)$ have isomorphic  Sylow $p$-subgroups.
\end{theorem}

This theorem is related to the following result of Holt \cite{Hol77}:

\begin{theorem}[\cite{Hol77}]
    \label{thm:holt}
    Let $G$ be a finite group and $P$ a Sylow $p$-subgroup of $G$. If the nilpotency class of $P$ does not exceed $p/2$, then the Schur multipliers of $G$ and $N_G(P)$ have isomorphic Sylow $p$-subgroups.
\end{theorem}

In fact, the majority of the proof of Theorem \ref{thm:bog} is an adaptation of Holt's argument. The crucial difference is the step where the bound on the nilpotency class of the Sylow $p$-subgroup is improved from Holt's $p/2$ in the Schur multiplier case, to $p$ in the Bogomolov multiplier case. Note that this does not improve the bound in Theorem \ref{thm:holt}, it merely shows that one can relax it when passing to Bogomolov multipliers.

The outline of the paper is as follows. We first provide some preliminaries in Section \ref{s:prelim}. Then we proceed to the proof of Theorem \ref{thm:bog}. To keep the exposition short, we only include the details where our argument differs from \cite{Hol77}, and refer to {\it loc. cit.} for the rest.
%%%%%%%%%%%%%%%%%%%%%%%%%%%%%%%%%%%%%%%%%%%%%%%%%%%%%%%%%%%%%%%%%%%%%%%
\section{Preliminaries}
\label{s:prelim}

\noindent
Most of the notations follow \cite{Bro82}. The maps and actions are always written from the right.

Let $G$ be a finite group. The second homology group $\HH^2(G,\QZ)$ is the {\it Schur multiplier} of $G$. If $1\to R\to F\to G\to 1$ is a free presentation of $G$, then $\HH^2(G,\QZ)$ is naturally isomorphic to $\Hom (M(G),\QZ)$, where $M(G)=(F'\cap R)/[R,F]\cong \HH_2(G,\Z)$, see, e.g., \cite[p. 42, p. 145]{Bro82}.

Given a $G$-module $M$, denote
$$\Sha^n(G,M)=\bigcap _{\tiny\begin{matrix}A\le G,\\ A \hbox{ abelian}\end{matrix}} \ker (\res ^G_A: \HH^n(G,M )\to \HH^n(A,M )).$$
Bogomolov \cite{Bog87} studied the group $\Sha^2(G,\QZ)$
in relationship with Noether's problem from invariant theory. This group is nowadays called the {\it Bogomolov multiplier} of $G$. It is shown in \cite[Proposition 3.8]{Mor12} that if $1\to R\to F\to G\to 1$ is a free presentation of $G$, then $\Sha^2(G,\QZ)$ is naturally isomorphic to $\Hom (\B_0(G),\QZ)$, where $B_0(G)=(F'\cap R)/\langle \K(F)\cap R\rangle$. Here $\K(F)$ stands for the {\it set} of all commutators $[x,y]$, where $x,y\in F$. 

Let a finite group $G$ act on a finite group $Q$. Then $G$ also acts on $\HH^2(Q,\QZ)$ via the rule $(c+\B^2(Q,\QZ))g=c'+\B^2(Q,\QZ)$, where the cocycle $c':Q\times Q\to\QZ$ is given by the rule $c'(q_1,q_2)=c(q_1g^{-1},q_2g^{-1})$. Via a free presentation $Q\cong F/R$ of $Q$, we have an action of $G$ on $M(Q)$, given as follows. Let $F$ be free on $X$. Take an isomorphism $\phi :F/R\to Q$, and let $g\in G$ and $x\in X$. Pick $y_x\in F$ with the property that $(xR)\phi g \phi^{-1}=y_xR$. This gives rise to an endomorphism $\psi :F\to F$ that sends $x$ to $y_x$. Note that $R$ is $\psi$-invariant. Thus $\psi$ induces an action of $G$ on $M(Q)$ via $(r[R,F])g=(r\psi)[R,F]$, where $r\in F'\cap R$.
It is not difficult to show that $\HH^2(Q,\QZ)$ and $\Hom(M(Q),\QZ)$ become isomorphic as $G$-modules.

The Bogomolov multiplier $\Sha^2 (Q,\QZ)$ is a subgroup of $\HH^2(Q,\QZ)$. Take $g\in G$ and $\alpha\in\Sha ^2(Q,\QZ)$. Take an arbitrary abelian subgroup $A$ of $Q$. Then $Ag^{-1}$ is abelian, hence $\alpha\res^G_{Ag^{-1}}=0$. By definition, $(\alpha g)\res^G_A=0$. Thus $\Sha^2(Q,\QZ)$ is a submodule of the $G$-module $\HH^2(Q,\QZ)$.

Let $\phi : F/R\to Q$ be a free presentation of $Q$, and let $\psi$ be as above. Denote $M_0(Q)=\langle \K(F)\cap R\rangle / [R,F]$. Take $g\in G$ and $x,y\in F$ with $[x,y]\in R$. Then $([x,y][R,F])g=([x,y]\psi )[R,F]=[x\psi,y\psi][R,F]$. As $R\psi\subseteq R$, it follows that $[x,y]\psi\in \K(F)\cap R$. This shows that $M_0(Q)$ is a submodule of the $G$-module $M(Q)$, hence $\B_0(Q)=M(Q)/M_0(Q)$ becomes a $G$-module. Similarly as above, $\Sha^2(Q,\QZ)$ and $\Hom (\B_0(Q),\QZ)$ are isomorphic as $G$-modules.

We will also need a couple of auxiliary results on commutator subgroups. The notations follow \cite[p. 119]{Rob82}.
\begin{lemma}
    \label{lem:comm}
    Let $A$ and $B$ be normal subgroups of a group $G$.
    \begin{enumerate}
        \item $[A,A,{}_nB]\le\prod\limits_{i=1}^n[[A,{}_iB],[A,{}_{n-i}B]]$ for all $n\ge 1$.
        \item $[\gamma_n(B),A]\le [A,{}_nB]$ for all $n\ge 1$.
    \end{enumerate}
\end{lemma}
    \begin{proof}
        The item (1) is proved in \cite[Lemma 6]{Hol77}. We prove (2) by induction on $n$. The case $n=1$ is clear. Suppose the claim holds for some $n\ge 1$ and all normal subgroups $A$ and $B$ of $G$. Then the Three Subgroup Lemma \cite[5.1.10]{Rob82} implies $[\gamma_{n+1}(B),A]=[\gamma_n(B),B,A]\le [B,A,\gamma_n(B)][A,\gamma_n(B),B]$. By induction assumption we have $[B,A,\gamma_n(B)]=[[B,A],\gamma_n(B)]\le [[B,A],{}_nB]=[A,{}_{n+1}B]$. In addition to that, 
        $[A,\gamma_n(B),B]=[[A,\gamma_n(B)],B]\le [A,{}_{n+1}B]$, again by induction assumption. This proves the result. 
    \end{proof}
%%%%%%%%%%%%%%%%%%%%%%%%%%%%%%%%%%%%%%%%%%%%%%%%%%%%%%%%%%%%%%%%%%%%%%%
\section{Proof of Theorem \ref{thm:bog}}
\label{s:proof}

\noindent
Let $H$ be a subgroup of a finite group $G$. Let $M$ be a $G$-module. Given $g\in G$, the conjugation map $H\to H^g$ induces an isomorphism $\conj^g_H: \HH^n(H,M)\to \HH^n(H^g,M)$. According to \cite{Hol77},  we say that $\alpha\in\HH^n(H,M)$ is {\it stable} with respect to $G$ (also {\it $G$-invariant} according to \cite{Bro82}) if $$\alpha\res^H_{H\cap H^g}=\alpha\conj^g_H\res^{H^g}_{H\cap H^g}$$ for every $g\in G$. The following is a crucial property of stable elements:

\begin{lemma}[\cite{Hol77}, Lemma 2]
    \label{lem:holtLemma2}
    Let $p$ be aprime, and let $H$ be a subgroup of $G$ of index coprime to $p$. Let $P_G$ and $P_H$ be the Sylow $p$-subgroups of $\HH^n(G,M)$ and $\HH^n(H,M)$, respectively. Then $P_G$ is isomorphic to the group of stable elements of $P_H$, which is a direct factor of $P_H$.
\end{lemma}

We first show that an analogous result holds for $\Sha^n$. Let $G$ and $H$ be as in Lemma \ref{lem:holtLemma2}. let $S_G$ and $S_H$ be the Sylow $p$-subgroups of $\Sha^n(G,M)$ and $\Sha^n(H,M)$, respectively. Then $S_G\le P_G$ and $S_H\le P_H$. Let $\rho$ be the restriction of the map $\res^G_H$ to $S_G$, and let $\sigma$ be the restriction of the map $\cor^G_H$ to $S_H$. Pick $\alpha\in\Sha ^n(G,M)$, and let $A$ be an abelian subgroup of $H$. Then $\alpha\res^G_H\res^H_A=\alpha\res^G_A=0$, hence $\alpha\res^G_H\in\Sha^n(H,M)$. This shows that $\rho$ maps $S_G$ into $S_H$. Similarly, if $\beta\in\Sha^n(H,M)$ and if $A$ is an abelian subgroup of $G$, then \cite[Proposition 9.5, Chapter III]{Bro82} gives
$$\beta\cor^G_H\res^G_A=\sum_{s\in\mathcal{S}}\beta\conj_H^s\res^{H^s}_{H^s\cap A}\cor^{A}_{H^s\cap A},$$
where $\mathcal{S}$ is a complete set of representatives of double cosets $HgA$, where $g\in G$. Note that $\beta\conj_H^s\in\Sha^n(H^s,M)$, and, as $H^s\cap A$ is an abelian subgroup of $H^s$, we get $\beta\conj_H^s\res^{H^s}_{H^s\cap A}=0$. Thus the above formula implies $\beta\cor^G_H\in\Sha^n(G,M)$. Hence $\sigma$ maps $S_H$ into $S_G$. As $|G:H|$ is not divisible by $p$, it follows from \cite[Proposition 9.5, Chapter III]{Bro82} that $\rho\sigma=|G:H|\cdot 1$ is an automorphism of $\Sha^n(G,M)$. Similarly as in \cite{Hol77} we can now show the following: 
\begin{lemma}
    \label{lem:myLemma2}
    Let $p$ be a prime, and let $H$ be a subgroup of $G$ of index coprime to $p$. Let $S_G$ and $S_H$ be the Sylow $p$-subgroups of $\Sha^n(G,M)$ and $\Sha^n(H,M)$, respectively. Then $S_G$ is isomorphic to to the group of stable elements of $S_H$, which is a direct factor of $S_H$.
\end{lemma}

Lemma \ref{lem:myLemma2} can be used to prove the following counterpart of \cite[Theorem 1]{Hol77}, with proof being a straightforward adaptation of Holt's argument:

\begin{corollary}
    \label{cor:myThm1}
    Let $G$ be a finite group and $P$ a Sylow $p$-subgroup of $G$. Let $M$ be a trivial $G$-module. Let $W$ be a characteristic $p$-functor which strongly controls fusion in $G$. Then the Sylow $p$-subgroups of $\Sha^n(G,M)$ and $\Sha^n(N_G(W(P)),M)$ are isomorphic.
\end{corollary}

Our next result is a key step in proving Theorem \ref{thm:bog}. 

\begin{proposition}
    \label{prop:Zc}
    Let $Q$ be a normal subgroup of a finite group $G$. Suppose that $[Q,{}_cG]=1$ for some $c\ge 1$. Then
    $[\Sha^2(Q,\QZ),{}_{c-1}G]=1$.
\end{proposition}

\begin{proof}
    It suffices to show that $[B_0(Q),{}_{c-1}G]=1$. Let $1\to R\to F\to G\to 1$ be a free presentation of $G$. Then $Q$ has a free presentation of the form $1\to R\to F_1\to Q\to 1$, where $F_1\trianglelefteq F$. By assumption we have that $[F_1,{}_cF]\le R$. Now Lemma \ref{lem:comm} (1) gives
    \begin{align*}
        [F_1'\cap R,{}_{c-1}F] &\le [F_1,F_1,{}_{c-1}F] \\
        &\le \prod_{i=1}^{c-1}[[F_1,{}_iF],[F_1,{}_{c-i-1}F]].
    \end{align*}
    Consider a commutator of the form
    $$\omega =[[x,a_1,\ldots ,a_i],[y,b_1,\ldots ,b_{c-i-1}]],$$
    where $x,y\in F_1$, $a_1,\ldots a_i,b_1, b_{c-i-1}\in F$, $1\le i\le c-1$.
    As $F_1$ is a normal subgroup of $F$, we have that $\omega\in\K(F_1)$. On the other hand, Lemma \ref{lem:comm} yields that $[[F_1,{}_iF],[F_1,{}_{c-i-1}F]]\le [F_1,{}_iF,\gamma_{c-i}F]\le [F_1,{}_cF]\le R$, therefore $\omega \in R$. This shows that
    $[F_1'\cap R,{}_{c-1}F]\le\langle \K(F_1)\cap R\rangle$, hence the result.
\end{proof}

The following result can be proved by essentially repeating the argument of the second part of the proof of \cite[Lemma 7]{Hol77}:

\begin{corollary}
    \label{cor:NG}
    Let $G$ be a finite group and $P$ a Sylow $p$-subgroup of $G$. Suppose that the nilpotency class of $P$ does not exceed $p$. Let $Q$ be a normal subgroup of $G$. If $Q$ is a $p$-group, then $[\Sha^2(Q,\QZ),G]=[\Sha_2(Q,\QZ),N_G(P)]$. 
\end{corollary}

\begin{proof}
    Form $H=G\ltimes \Sha^2(Q,\QZ)$. Then $S=P\ltimes \Sha^2(Q,\QZ)$ is a Sylow $p$-subgroup of $H$. As $P$ is nilpotent of class $\le p$, Proposition \ref{prop:Zc} implies $[\Sha^2(Q,\QZ),{}_{p-1}P]=1$. The rest of the proof now follows the lines of \cite[Proof of Lemma 7]{Hol77}.
\end{proof}

Having Corollary \ref{cor:NG} at hand, we finish the proof of Theorem \ref{thm:bog} by applying  \cite[Lemma 8]{Hol77} and repeating the argument following it.
%%%%%%%%%%%%%%%%%%%%%%%%%%%%%%%%%%%%%%%%%%%%%%%%%%%%%%%%%%%%%%%%%%%%%%%

\end{document}